\documentclass[10pt]{article} 

\usepackage[utf8]{inputenc} 

\usepackage[margin=1in]{geometry} 
\usepackage{graphicx} 

\usepackage{booktabs} 
\usepackage{array} 
\usepackage{paralist} 
\usepackage{verbatim} 
\usepackage{subfig} 
\usepackage{mathrsfs}
\usepackage{amssymb}
\usepackage{amsthm}
\usepackage{amsmath,amsfonts,amssymb,esint}
\usepackage{graphics}
\usepackage{enumerate}
\usepackage{mathtools}
\usepackage{xfrac}
\usepackage{fancyhdr} 
\usepackage[nottoc,notlof,notlot]{tocbibind} 
\usepackage[titles,subfigure]{tocloft} 

\numberwithin{equation}{section}

\def\XXint#1#2#3{{\setbox0=\hbox{$#1{#2#3}{\int}$ }
\vcenter{\hbox{$#2#3$ }}\kern-.6\wd0}}

\theoremstyle{plain}\newtheorem{thm}{Theorem}[section]
\theoremstyle{plain}\newtheorem{prop}[thm]{Proposition}
\theoremstyle{plain}\newtheorem{lem}[thm]{Lemma}
\theoremstyle{plain}
\theoremstyle{plain}\newtheorem{remark}[thm]{Remark}
\theoremstyle{plain}
\theoremstyle{plain}

\newcommand {\sys}{$(\mathcal {S-BC}_e)$}

\newcommand {\sysss}{$(\mathcal {S}_{da}\mathcal{-BC}_{e})$}
\newcommand{\nr}{{|\!|}}

\newcommand{\st}{\int_0^t}

\newcommand{\s}{\int_{\Omega}\,}
\newcommand{\R}{{\mathbb{R}^+}}
\newcommand{\0}{{L^2}}

\usepackage[usenames,dvipsnames]{color}
\usepackage[colorlinks=true, pdfstartview=FitV, linkcolor=blue, citecolor=blue, urlcolor=blue,pdfencoding=auto]{hyperref}

\title{\bf Global well-posedness of a three-dimensional Brinkman-Forchheimer-B\'enard convection model in porous media}

\author{
{   {\bf Edriss S. Titi}}\thanks{\footnotesize Department of Mathematics, Texas A\&M University, College Station, TX 77843, USA; and Department of Applied Mathematics and Theoretical Physics, University of Cambridge,
Cambridge CB3 0WA, UK.
E-mail address: titi@math.tamu.edu, \:Edriss.Titi@maths.cam.ac.uk}
 \and
 {\!\!  {\bf Saber Trabelsi}}\thanks{\footnotesize Science Program, Texas A\&M University at Qatar, P.O. Box 23874 Doha, Qatar. E-mail address: saber.trabelsi@qatar.tamu.edu  }
}

\date{April 7,2022}

\begin{document}

\maketitle
{\it This work is dedicated to Professor Jerome A. Goldstein on the occasion of his $80^{th}$ birthday}

\begin{abstract}
We consider three-dimensional (3D) Boussinesq convection system of an incompressible fluid in a closed sample of  a porous medium. Specifically, we introduce and analyze a 3D Brinkman-Forchheimer-B\'enard convection problem describing the behavior of an incompressible fluid in a porous medium between two plates heated from the bottom and cooled from the top. We show the existence and uniqueness of global in-time solutions, and the existence of absorbing balls in $L^2$ and $H^1$. Eventually, we comment on the applicability of a data assimilation algorithm to our system.
\end{abstract}
\vskip6pt
\noindent \textbf{MSC class:} 35Q30, 35Q35, 76B03, 86A10, 93C20, 37C50, 76B75, 34D06.
\setcounter{tocdepth}{2}
\tableofcontents
\section{Introduction and main results}
In its original description, the B\'enard convection problem is concerned with the motion of incompressible flow confined between two horizontal plates (or walls) heated at the bottom and cooled at the top. Density differences occur due to the temperature difference across the fluid as regions of the fluid near the bottom boundary are heated and thus expand. In turn, the density differences result in a buoyancy force pushing the lighter fluid to the top and the heavier to the bottom and the governing equation of motion is modeled with the Boussinesq approximation. Applications of B\'enard convection range from weather forecasting to nuclear magnetic resonance pulsed-field-gradient diffusion measurements \cite{app1} and the security of liquefied natural gas packets \cite{app2} etc. In this paper we analyze the B\'enard problem in a porous medium modeled by the Brinkman-Forchheimer extended Darcy system for the momentum equation coupled with the heat convection.
\subsection{The physical model}
In this paper, we consider the B\'enard convection problem of an incompressible fluid  saturating an infinite horizontal layer of porous medium confined between two horizontal solid walls located at $z=0$ and $z=1$. The fluid is heated from below and cooled from the top with temperatures normalized to $1$ and $0$, respectively.  Using the Boussinesq approximation, the non-dimensional 3D equations governing the motion of the convected fluid through the porous medium are given by the Brinkman-Forchheimer-B\'enard  system
\begin{align*}
\mathcal {S}_o :\quad \left\lbrace\begin{array}{ll}
&\partial_t\,u-\nu\,\Delta\,u + (u\cdot \nabla)\,u+a\,|u|^{2\alpha}u+\nabla q= T\,{\bf e}_3,\\ \\
&\partial_t\,T-\kappa\,\Delta\,T +(u\cdot \nabla) T = 0,\\ \\
& \nabla\cdot u=0,\;u_{\vert_{t=0}}=u_0,\; T_{\vert{{t=0}}}=T_0,
\end{array}
\right.
\end{align*}
where ${\bf e}_3=(0,0,1)^T$. We consider this problem in an horizontal periodic domain  $\Omega_r:= [0,L]\times [0,L] \times[0,1]$\footnote{Obviously the domain can be chosen as $ [0,L_1]\times [0,L_2] \times[0,L_3]$ with $L_1,L_2,L_3>0$ and  $L_1\neq L_2\neq L_3$.}, and supplement it with the following boundary conditions (BC in short)
\begin{align*}
\mathcal {BC}_o:\quad \left\lbrace\begin{array}{ll}
&T(t,x,y,0)=1,\quad T(t,x,y,1)=0, \quad u_3(t,x,y,0)=u_3(t,x,y,1)=0,\\ \\
& \partial_3 u_1 (t,x,y,0)= \partial_3 u_1 (t,x,y,1)= \partial_3 u_2 (t,x,y,0)= \partial_3 u_2 (t,x,y,1)=0,\\ \\
&q,u,T, \text{periodic in the $x$ and $y$ variables with period $L$.}
\end{array}
\right.
\end{align*}
In system $(\mathcal {S}_o-\mathcal{BC}_o)$, the fluid velocity $u(t,x,y,z)$, the pressure $p=p(t,x,y,z)$, and the normalized temperature $T=T(t,x,y,z)$ are the unknowns. $\nu$ and $\kappa$ are positive constants representing the kinematic viscosity and the thermal diffusivity respectively, and $a$ is a positive coefficient that arises from the Darcy-Forcheheimer law. In this model we use the Brinkman-Forchheimer-extended Darcy (BFeD in short) model for flow in porous media, instead of the simple Darcy law. The velocity BC in $\mathcal {BC}_o$  are no-normal flow and stress-free at the solid boundary.
\vskip6pt
\noindent
Observe that $u=0, T=1-z$ and $q=z\left(1-\frac{z}2\right)$ is the pure conduction steady state of $(\mathcal {S}_o-\mathcal{BC}_o)$. Considering a fluctuation around this steady state
\begin{equation}\label{change}
\theta = T-(1-z)\quad \text{and}\quad p=q-z\left(1-\frac{z}2\right),
\end{equation}
then system $(\mathcal {S}_o-\mathcal{BC}_o)$ is equivalent to
\begin{align*}
\mathcal S :\quad \left\lbrace\begin{array}{ll}
&\partial_t\,u-\nu\,\Delta\,u + (u\cdot \nabla)\,u+a\,|u|^{2\alpha}u+\nabla p= \theta \,{\bf e}_3,\\ \\
&\partial_t\,\theta-\kappa\,\Delta\,\theta +(u\cdot \nabla) \,\theta = u\cdot \,{\bf e}_3,\\ \\
& \nabla\cdot u=0,\;u_{\vert_{t=0}}=u_0,\; \theta_{\vert{{t=0}}}=\theta_0,
\end{array}
\right.
\end{align*}
supplemented with the corresponding set of  BC; obtained from $\mathcal {BC}_o$ using \eqref{change}. More precisely
\begin{align*}
\mathcal {BC}:\quad\left\lbrace\begin{array}{ll}
&\theta(t,x,y,0)= \theta(t,x,y,1)=0,\quad u_3(t,x,y,0)=u_3(t,x,y,1)=0,\\ \\
& \partial_3 u_1 (t,x,y,0)= \partial_3 u_1 (t,x,y,1)= \partial_3 u_2 (t,x,y,0)= \partial_3 u_2 (t,x,y,1)=0,\\ \\
&p,u,\theta, \text{periodic in the $x$ and $y$ variables with respective periods $L$.}
\end{array}
\right.
\end{align*}
In the case of $a=0$, the first set of equations of system $\mathcal S$ ($\mathcal S_1$ in short) is nothing but the classical 3D Navier-Stokes equations forced by bouyancy, so that $\mathcal S$ corresponds to the classical 3D  Boussinesq equations. In this case, the mathematical analysis of the B\'enard system $(\mathcal {S}-\mathcal{BC})$ has been studied in \cite{FoITmanem} (see also \cite{Temam1} and references therein). The authors prove the existence and uniqueness of weak solutions in two-dimensional  space (2D), and the existence of weak solutions in 3D. Also, they proved  the existence of a finite-dimensional global attractor in 2D. Let us mention that the authors used the third line of $\mathcal {BC}$ and Dirichlet for $u$ and $\theta$ at the top and the bottom boundaries as BC.
\vskip6pt
\noindent
When $a>0$, equations $\mathcal S_1$ (with $\theta\equiv 0$) are the so-called the 3D BFeD model. This model was formally derived (cf., e.g., \cite{Hsu}) using Darcy-Forchheimer equation of porous media that states
\[ \nabla p= -\frac{\mu}{k}\,{\bf v}_f- \gamma \rho_f |{\bf v}_f|^2 \,{\bf v}_f,\]
where $\gamma>0, {\bf v}_f$ and $\rho_f$  stand for the the Forchheimer coefficient, the Forchheimer velocity and the density, respectively.  This equation add a correction to the Darcy law to model the increase of the pressure drop.  There is a rich literature dedicated to the mathematical analysis of this model and its variants, and we refer to, e.g., \cite{3,12,zbMATH06433791,zbMATH06840730,18,21,2,17,23,24,YCL,Cai,Oliveira}. Recently,  In \cite {Varga, MTT}, the  authors shows the existence and uniqueness of weak and strong solutions with Dirichlet boundary condition starting from a regular enough initial data. In the periodic setting, the authors of  \cite{MTT} improve the results of \cite{Varga} and prove the well posedness for initial data in $H^1(\mathbb T)$. Their result can be extended to the case of Dirichlet boundary condition using  the regularity estimates of the Stokes operator). An anisotropic viscous version of the BFeD system was studied in \cite{Hakima2}. Eventually, a relatively closed (from the mathematical point of view) MHD model was investigated in \cite{titi2018global} and a its Boussinesq-MHD version (without diffusion) in \cite{liu2019global}. Let us mention that in the latter reference, the uniqueness was obtained only for regular solutions, and our argument in the present contribution  combined with ideas from \cite{titi2018global} can improve the result.
\vskip6pt
\noindent
To overcome technicalities related to the boundary conditions, we extend the domain $\Omega_r$ to $\Omega=  [0,L]\times [0,L] \times[-1,1]$ and consider problem $\mathcal S$, subjected to the following set of BC
\begin{align*}
\mathcal {BC}_e:\quad \left\lbrace\begin{array}{ll}
& \theta, u_3,  \,\text{periodic odd functions with respect to the $z$ variable with period $2$,}\\ \\
&u_1,u_2, \, \text{periodic even functions with respect to the $z$ variable with period $2$,}\\ \\
&p,u,\theta, \text{periodic in the $x$ and $y$ variables with period $L$}.
\end{array}
\right.
\end{align*}
It is rather easy to see that this set of  periodic-symmetric BC are equivariant under the solution operator of system  $\mathcal S$ supplemented with periodic conditions $\mathcal{BC}$. Most importantly, solutions with this periodic-symmetric BC $\mathcal {BC}_e$ clearly satisfy the physical BC $\mathcal{BC}$. Indeed, the fact that  $ \theta, u_3$ are periodic odd functions with respect to the $z$ variable with period $2$ implies that $\theta(t,x,y,1)=0, u_3(t,x,y,1)=0$. Equivalently, the fact that $u_1,u_2$ are periodic even functions with respect to the $z$ variable with period $2$ implies that  $\partial_3 u_1 (t,x,y,1)= \partial_3 u_2 (t,x,y,1)=0$. As a matter of fact, we shall focus on the mathematical analysis of system  $(\mathcal {S}-\mathcal{BC}_e)$. Obviously all the results obtained in the periodic boundary conditions setting will be valid for the physical system $(\mathcal {S}-\mathcal{BC})$ in $\Omega_r$. In other words,  the restriction to $\Omega_r$ of a solution $(u(x,y,z),\theta(x,y,z),p(x,y,z))$ of system $(\mathcal {S}-\mathcal{BC}_e)$  on $\Omega$, is a solution of $(\mathcal {S}-\mathcal{BC})$.
\subsection{The Main results}
Let us introduce the functional setting that we shall use along this paper. Let $\mathcal X_e$ be the set of trigonometric polynomials with period $L$  in the $x$ and $y$ variables, and are even with period $2$ in the $z$ variable. Let $\mathcal X_o$ be the set of trigonometric polynomials with period $L$ in the $x$ and $y$ variables, and are odd with period $2$ in the $z$ variable. Eventually, let $\mathcal Y$ be the set of divergence-free vector fields belonging to  $\mathcal X_e\times\mathcal X_e\times \mathcal X_o$. In the sequel, we will not make a difference in the notation of scalar and vector Lebesgue and Sobolev spaces, which shouldn't confuse the reader.
\vskip6pt
\noindent
Now, we define $H_0$ and $H_1$ as the closure of $\mathcal Y$ and $\mathcal X_o$ in $L^2(\Omega)$, respectively. We endow  $H_0$ and $H_1$ with the following scalar products

\[(u,v)_{H_0} = \sum_{i=1}^3\,\s u_i(x)\,v_i(x)\,dx,\quad \text{ and } \quad ( \varphi,\phi)_{H_1} = \s\varphi(x)\,\phi(x)\,dx. \]
The associated norms are given by $\nr u\nr_{H_0}= \left[(u,u)_{H_0}\right]^\frac12$ and $\nr \varphi\nr_{H_1}= \left[(\varphi,\varphi)_{H_1}\right]^\frac12$, respectively. Equivalently, we define $V_0$ and $V_1$ as the closure of $\mathcal Y$ and $\mathcal X_o$ in $H^1(\Omega)$, respectively. $V_0$ and $V_1$ are Hilbert spaces endowed with the following scalar products
\[(u,v)_{V_0} =(u,v)_{H_0}+ (u,v)_{\dot V_0}:= (u,v)_{H_0}+\sum_{i,j=1}^3\,\s\partial_j\,u_i(x)\,\partial_j\,v_i(x)\,dx,\]
and
\[(\varphi,\phi)_{V_1} = \sum_{j=1}^3\,\s\partial_j\,\varphi(x)\,\partial_j\,\phi(x)\,dx,\]
where $\partial_j$ denotes the partial derivative with respect to the variable $x$ if $j=1$, $y$ if $j=2$, and $z$ if $j=3$. The associated norms are given by  $\nr u\nr_{V_0}= \left[(u,u)_{V_0}\right]^\frac12$ and $\nr \varphi\nr_{V_1}= \left[(\varphi,\varphi)_{V_1}\right]^\frac12$, respectively. It is worth noticing that since $\theta, u_3 \in V_1$ are odd in the $z$ variable and periodic in the $x$ and $y$ variables, they have average zero over $\Omega$, thus by the Poincar\'e inequality, $\nr\cdot\nr_{V_1}$   defines a norm on $V_1$.
\vskip6pt

\noindent In the sequel, we shall  use the notation $\nr\cdot\nr_{p}$ for the $L^p(\Omega)$ norms, and $\nr \cdot\nr_{H^1}$ and $\nr \cdot\nr_{H^2}$  for $H^1(\Omega)$ and $H^2(\Omega)$ norms respectively. Now, let ${\mathcal A}_i$, for $i=0,1$, be the unbounded nonnegative self-adjoint linear operators with domains $D({\mathcal A}_i)=V_i \cap H^2(\Omega)$ satisfying $({\mathcal A}_i \varphi,\phi)_{H_i} = ((\varphi,\phi))_{V_i}$ for all $\varphi,\phi \in D({\mathcal A}_i)$ and $i=0,1$. The operator ${\mathcal A}_1$ is positive definite with compact inverse ${\mathcal A}_1^{-1}$. Observe that with periodic BC, we have $\mathcal A_0=-\Delta$, which is not invertible whose kernel consists of constant vector fields corresponding to the eigenvalue $0$. Thanks to the elliptic regularity of the operator $\mathcal A_0+I$ and Cauchy-Schwarz inequality, it is rather easy to see that $\nr u \nr_{H^2}\simeq\nr u\nr_{L^2}+\nr A_0 u\nr_{L^2}$. Consequently, there exists a basis of orthonormal eigenfunctions $w_{i,j}$ of $\mathcal A_i$ for $i=0,1$ and $j=1,2,\ldots$ such that ${\mathcal A}_iw_{i,j}=\lambda_{i,j}\,w_{i,j}$ where  $w_{i,j} \in H_i$ denotes the $j^{th}$ eigenfunction of ${\mathcal A}_i$ and $\lambda_{i,j}$ the associated positive eigenvalue satisfying  $0<\lambda_{i,j}\leq \lambda_{i,j+1}$ for all $i=0,1$ and $j=1,2,\ldots$ Let us mention that, by abuse of notation, we denote $\lambda_{0,1}$ the second eigenvalue of $\mathcal A_0$  since the first eigenvalue is $0$ as it was stated above, and we add the associated eigenvector to the basis spanning the kernel of $\mathcal A_0$. Thus, introducing  $\lambda:= \inf_{i=0,1; j =1,2,\ldots}\,\lambda_{i,j}>0$, the Poincar\'e and Sobolev inequalities read $ \lambda^{1/2}\nr \varphi\nr_{2} \leq \nr \nabla \varphi\nr_2$ and $\nr \varphi\nr_6 \leq \gamma\, \nr \nabla \varphi\nr_{2}$, respectively, where $\varphi=\theta, u_3,\partial_3 u_1,\partial_3 u_2, \partial_1 u_3, \partial_2 u_3$ thanks to the set of symmetric periodic BC, $\mathcal {BC}_e$ (with $\gamma>0$ being a constant depending only on the size of the domain).

\vskip6pt
\noindent Now, we are able to state our first result about the existence of solutions to system \sys
\begin{thm} \label{thmweak}
Let $(u_0,\theta_0)\in H_0 \times H_1, \alpha\geq 0$ and $a,\nu,\kappa> 0$, then system  \sys\, has global weak solutions $(u,\theta)$ satisfying
\begin{align*}
&u(x,t) \in C_b^0(\mathbb R^+;H_0)\cap L^2_{\rm loc}(\mathbb R^+; V_0) \cap L^{{2\alpha+2}}_{\rm loc}(\mathbb R^+;L^{{2\alpha+2}}(\Omega)),\\&{\rm and}\\
& \theta(x,t) \in C_b^0(\mathbb R^+;H_1)\cap L^2_{\rm loc}(\mathbb R^+; V_1).
\end{align*}
In particular
\begin{equation}\label{limsupthm} \limsup_{t\to+\infty}\,\nr \theta(t) \nr_{H_1}, \quad\limsup_{t\to+\infty}\,\nr u(t) \nr_{H_0} \leq \frac{4a\,L^2}{\min{(a,\kappa\lambda})}\, \left(\frac{2a\kappa\lambda}{a\kappa\lambda +32}\right)^{\frac{\alpha+1}{\alpha}}.
\end{equation}
In addition, if $\alpha>1$, then the weak solutions depend continuously on the initial data in the $H_0 \times H^{-1}(\Omega)$ topology, in particular they are unique.
\end{thm}
\noindent This Theorem ensures the existence of weak solutions to system \sys\, and their uniqueness for  a damping parameter's range $\alpha>1$. Moreover, it shows the existence of absorbing ball in $H_0\times H_1$ for the solutions of \sys\,. This property plays a crucial role in the design and  analysis of a data assimilation (DA) algorithm for system \sys\, (see section \ref{dasec}). Also, we have the following
\begin{thm} \label{thmstrong}
Let $(u_0,\theta_0)\in V_0 \times H_1, \alpha > 1$ and $a,\nu,\kappa>0$, then system  \sys\, has global  solutions $(u,\theta)$ satisfying
\begin{align*}
&u(x,t) \in C_b^0(\mathbb R^+;V_0)\cap L^2_{\rm loc}(\mathbb R^+; V_0\cap H^2(\Omega)) \cap L^{{2\alpha+2}}_{\rm loc}(\mathbb R^+;L^{{2\alpha+2}}(\Omega)),\\&{\rm and}\\
& \theta(x,t) \in C_b^0(\mathbb R^+;H_1)\cap L^2_{\rm loc}(\mathbb R^+; V_1),
\end{align*}
In particular, in addition to \eqref{limsupthm}, we have
\begin{align}\label{limsupgradthmstraong}
\limsup_{t\to+\infty} \,\nr  u(t)\nr^2_{\dot V_0} &\leq\frac{\Gamma_2\left[(a+1)\Gamma_1 + 4aL^2\right]+(3+a)\Gamma_1 + 4aL^2}{2\nu},
\end{align}
where
\[
 \Gamma_1:=\left(\frac{2a\kappa\lambda}{a\kappa\lambda +32}\right)^{\frac{\alpha+1}{\alpha}}\,\frac{4a\,L^2}{\min{(a,\kappa\lambda})} \quad {\text{and}}\quad \Gamma_2:=\left(\frac{a\nu^\alpha }{2^{2-\alpha}}\right)^{1/1-\alpha}.
\]
Moreover,  if $\theta_0\in L^6(\Omega)$, then the  solutions are unique. Furthermore, if $u_0\in V_0\cap L^{2\alpha+2}(\Omega)$, then $u\in L^{\infty}_{\rm loc}(\mathbb R^+;L^{{2\alpha+2}}(\Omega))$ and $\partial _t u\in L^2_{\rm loc}(\mathbb R^+; H_0)$. Also, the solutions depend continuously on the initial data in the $H_0 \times H^{-1}(\Omega)$ topology, in particular they are unique.
\end{thm}
\noindent This Theorem shows the existence of global solutions with regular initial velocity. For a given temperature $\theta\in H_1$, the velocity solution $u\in V_0$ is a classical solution of the BFeD equations with the forcing term $\theta\,{\bf e}_3$. Equivalently, for a given velocity $u\in V_0$, the solution $\theta$ of the thermal diffusion equation is a weak solution. Of course, if the initial temperature is considered in $V_1$, then we can extend the Theorem and reach strong solutions for both velocity and temperature (see Remark \ref{strongtheta}).  Most importantly, this Theorem shows the existence of an absorbing ball  in $V_0\times H_1$ for the solutions of \sys. Obviously, the Theorem still holds if the initial temperature is in $L^\infty(\Omega)$, which corresponds to the physical case. Also, if one consider initial temperature in $V_1$, then system \sys\, have an absorbing ball  in $V_0\times V_1$. Eventually, let us mention that in Theorems \ref{thmweak} and \ref{thmstrong}, we do not discuss the regularity of the pressure which can be recovered from the velocity in a classical way using the divergence free property and standard elliptic regularity. We refer to any textbook for details about this point (i.e., \cite{sohr2012navier,Temam})
{\center {\it This work is dedicated to Professor Jerome A. Goldstein on the occasion of his 80$^{\text{\it th}}$ birthday as a token of admiration for his contribution to the mathematical analysis of Partial Differential Equations and their applications.}}
\section{Proof of well-posedness}
First, let us recall the following version of Young's inequality,
\[ab\leq \epsilon a^p + \epsilon^{-q/p} b^q,\quad \frac1p+\frac1q=1, \quad \text {for all} \quad \epsilon>0 \quad \text{and} \quad a,b\geq 0.\]

\subsection{Galerkin approximation system}\label{locexistence}
The well posedness of \sys\, can be shown using a standard approximation argument. First, one uses the Faedo-Galerkin approximation method based on an orthonormal basis of eigenfuctions of the operators $A_0$ and $A_1$ (see, e.g. \cite{Temam,Temam1}) to show the existence and uniqueness of approximate solutions. Next, one obtains uniform {\it a priori} estimates using the approximate system, and eventually pass to the limit using compactness arguments, e.g., Aubin-Lions  Lemma, \cite{Lions} (Lemma I-6.5).
\vskip6pt
\noindent
First, we define the bilinear forms $B_0(\cdot,\cdot): V_0\times \mathcal D(A_1) \longrightarrow H_0$ and $B_1(\cdot,\cdot): V_1\times \mathcal D(A_1) \longrightarrow H_1$ such that
\[B_0(u,v):=\mathbb P\,(u\cdot\nabla)v,\quad \text{and}\quad B_1(u,\theta):=(u\cdot\nabla)\theta.\]
\vskip6pt
\noindent where $\mathbb P$ denotes the Leray projector on divergence-free vector fields. Let  $\{\mathcal W_k(x)\} \subset D(\mathcal A_0)$ and $\{\mathcal {\tilde W}_k(x)\}\subset D(\mathcal A_1)$ be orthonormal basis of $H_0$ and $H_1$, consisting  of eigenfunctions of $\mathcal A_0$ and $\mathcal A_1$, respectively. Denote by
$V^m_0=\hbox {span} \{\mathcal W_1,\mathcal  W_2,\ldots, \mathcal  W_m\}$ and $V^m_1=\hbox {span}\{\mathcal {\tilde W}_1, \mathcal {\tilde W}_2,\ldots,  \mathcal {\tilde W}_m\}$, and $\mathcal P_m:H_0\to V_0^m$ and $\mathcal {\tilde P}_m:H_1\to V_1^m$ be the corresponding the projections. Now, set
\begin{align*}
&{  u}_m(x,t)=\sum_{i=1}^m g_m^i(t)\,\mathcal W_i(x) \in V^m_0 \quad \text{ and} \quad  {  \theta}_m(x,t)=\sum_{i=1}^m \tilde g_m^i(t)\,\mathcal {\tilde W}_i(x) \in V_1^m.
\end{align*}
by the solution of the Galerkin approximation system associated with \sys, namely, the unknown coefficients $g_m^i= ( u_m, \mathcal W_i)_{H_0}$ and $\tilde g_m^i= (   \theta_m, \mathcal {\tilde W}_i)_{H_1}$, for $i=1,2, \ldots, m$, solve  the following system of ordinary differential equations
\begin{subequations}\label{eqpart}
\begin{alignat}{3}
\label{eqpart1}& \frac d{dt} u_m +  \nu\,\mathcal A_0\,u_m +\mathcal P_m\,B_0(u_m,u_m) +a\,\mathcal P_m\,\left( \mathbb P\,\left(|u_m|^{2\alpha}u_m\right) \right)  -\mathcal P_m\,\left(\mathbb P\,\left(\theta_m\,{\bf e}_3\right)\right)=  0,\\
&  \frac d{dt} \theta_m + \kappa\,\mathcal A_1\,\theta_m +\mathcal {\tilde P}_m\,B_1(u_m,\theta_m)-\mathcal {\tilde P}_m\, \left(u_m\cdot {\bf e}_3\right) =0,\label{eqpart2}\\
\label{eqpart3}  &u_m(0)=\mathcal P_m u_0, \quad \theta_m(0)=\mathcal {\tilde P}_m\theta_0,
\end{alignat}
\end{subequations}
where  $(u_0,\theta_0)\in H_0\times H_1$. Observe that for $(u_0,\theta_0)\in H_0\times H_1$ one has
\begin{equation}\label{conn}
(\mathcal P_m{u_0},\mathcal {\tilde P}_m{\theta_0}) \longrightarrow ({u_0},{\theta_0}) \quad \text{ strongly in }H_0\times H_1\quad \text{as} \quad m \rightarrow +\infty.
\end{equation}
Since the vector field in system \eqref{eqpart} is locally Lipschitz in $ V_0^m \times V_1^m$, the system admits a unique solution $(u_m,\theta_m)\in C^1([0,\tau_m],V_0^m)\times C^1([0,\tau_m],V_1^m)$, for some   $\tau_m >0$.

\subsection{A priori estimates and existence of weak solutions}
Let $\mathcal T \in (0,\infty)$ be arbitrary. Our goal is to show that the unique solution of the Galerkin approximation system \eqref{eqpart} exists on the interval $[0, \mathcal T]$. Let $[0,\tau_m^*)$ is the maximal interval existence of solutions to \eqref{eqpart}, and assume by contradiction that  $\tau_m^* \le \mathcal T <\infty$. This in turn implies that

\begin{equation}\label{limsup}
\limsup_{t\to (\tau_m^*)^-} \Big(\nr u_m (t)\nr^2_{H_0} +\nr \theta_m (t)\nr^2_{H_1}\Big) = \infty.
\end{equation}
Next, we focus on the interval $[0,\tau_m^*)$ and establish {\it a priori} estimates for the solution of  \eqref{eqpart}. Taking the $H_0-$inner product of equation (\ref{eqpart}) with $u_m(t)$ and   thanks to Cauchy-Schwarz and Young inequalities one obtains
\begin{align}\label{galer1}
\frac12\,\frac{d}{dt}\, \nr u_m \nr^2_{H_0}+\nu\,\nr u_m\nr^2_{\dot V_0} +a\,\nr u_m\nr^{2\alpha+2}_{2\alpha+2}&  \leq \frac1{2}\,\nr \theta_m\nr^2_{H_1}+\frac{1}{2}\,\nr u_m\nr^2_{H_0}.
\end{align}
Similarly, we take the $H_1-$inner produce of equation (\ref{eqpart})$_2$  with $\theta_m(t)$  to obtain
\begin{align}\label{galer2}
\frac12\,\frac{d}{dt}\, \nr \theta_m \nr^2_{H_1}+\kappa\,\nr \theta_m\nr^2_{V_1} &  \leq \frac1{2}\,\nr \theta_m\nr^2_{H_1}+\frac{1}{2}\,\nr u_m\nr^2_{H_0}.
\end{align}
Summing \eqref{galer1} and \eqref{galer2}, we get for all $t\in [0,\tau_m^*)$
\begin{align}\label{allinall0}
\frac12\,\frac{d}{dt}\, \nr u_m \nr^2_{H_0}+\frac12\,\frac{d}{dt}\, \nr \theta_m \nr^2_{H_1}+\nu\,\nr u_m\nr^2_{\dot V_0}+\kappa\,\nr \theta_m\nr^2_{V_1} +a\,\nr u_m\nr^{2\alpha+2}_{2\alpha+2}&  \leq \nr \theta_m\nr^2_{H_1}+\nr u_m\nr^2_{H_0}.
\end{align}
Now, using Gronwall's inequality, we obtain
\begin{align*}
 \nr u_m(t) \nr^2_{H_0}+ \nr \theta_m(t) \nr^2_{H_1} \leq \left(\nr \mathcal {P}_m{u_0}\nr^2_{H_0}+\nr {\mathcal {\tilde P}_m \theta_0}\nr^2_{H_1}\right)\,e^{2t}
 \leq \left(\nr {u_0}\nr^2_{H_0}+\nr \theta_0\nr^2_{H_1}\right)\,e^{2t}.
\end{align*}
Therefore, integrating the inequality \eqref{allinall0}, we get
\begin{align*}
 \nr u_m(t) \nr^2_{H_0}+ \nr \theta_m(t) \nr^2_{H_1}&+2\nu\,\st\nr u_m(s)\nr^2_{\dot V_0}\,ds+2\kappa\,\st\nr \theta_m(s)\nr^2_{V_1} \,ds+2a\,\s\nr u_m(s)\nr^{2\alpha+2}_{2\alpha+2}\,ds\\&  \leq \left(\nr\mathcal {P}_m {u_0}\nr^2_{H_0}+ \nr \mathcal {\tilde P}_m{\theta_0}\nr^2_{H_1}\right)\,(1+e^{2t})\leq  \left(\nr {u_0}\nr^2_{H_0}+\nr {\theta_0}\nr^2_{H_1}\right)\,(1+e^{2t}),
\end{align*}
for all $t\in [0,\tau_m^*)$. The above contradicts  \eqref{limsup}. Therefore, the solution exists on $[0,\mathcal T]$. Moreover,  we also conclude from the above  that $(u_m,\theta_m)$ remains bounded uniformly in $m$ in
\begin{equation}\label{forconvergence} L^\infty([0,{\mathcal T}],H_0) \times L^\infty([0,{\mathcal T}],H_1) \cap L^2([0,{\mathcal T}],V_0) \times L^2([0,{\mathcal T}],V_1).
\end{equation}
Furthermore, $u_m$ remains bounded in $L^{2\alpha+2}([0,{\mathcal T}], L^{2\alpha+2}(\Omega))$ so that $|u_m|^{2\alpha}u_m$ remains bounded, uniformly in $m$, in
\begin{equation}\label{tousejustbelow} L^{\frac{2\alpha+2}{2\alpha+1}}([0,{\mathcal T}], L^{\frac{2\alpha+2}{2\alpha+1}}(\Omega)),\quad \text{since} \quad \nr |u_m|^{2\alpha}u_m\nr_{\frac{2\alpha+2}{2\alpha+1}}^{\frac{2\alpha+2}{2\alpha+1}}= \nr u_m\nr^{2\alpha+2}_{2\alpha+2}.
\end{equation}
Now, from system \eqref{eqpart} and by virtue of   \eqref{tousejustbelow} it is rather standard to see that $ (\frac{d}{dt} u_m,\frac{d}{dt} \theta_m)$ remains bounded, uniformly in $m$, in
\begin{equation*}
 L^2([0,{\mathcal T}], V_0^{'})+ L^{\frac{2\alpha+2}{2\alpha+1}}([0,{\mathcal T}],L^{\frac{2\alpha+2}{2\alpha+1}}(\Omega) ) \times L^2([0,{\mathcal T}], V_1^{'}).
\end{equation*}
Now, thanks to \eqref{forconvergence}, we infer that there exist $u$ and $\theta$, and there exist two subsequences, that we still denote $(u_m,\theta_m)$, such that
\begin{align}\label{con1}
(u_m,\theta_m) \rightarrow (u,\theta) \quad \text{ weakly}-* \quad \text{in } \quad L^\infty([0,{\mathcal T}],H_0) \times L^\infty([0,{\mathcal T}],H_1) \quad \text{as} \quad m \rightarrow +\infty,
\end{align}
and
\begin{align}\label{con2}
(u_m,\theta_m) \rightarrow (u,\theta) \quad \text{ weakly in } \quad L^2([0,{\mathcal T}],V_0) \times L^2([0,{\mathcal T}],V_1) \quad \text{as} \quad m \rightarrow +\infty,
\end{align}
 Furthermore, thanks to \eqref{tousejustbelow}, it holds
 \begin{align}\label{con3}
u_m \rightarrow u \quad \text{ weakly in } \quad L^{2\alpha+2}([0,{\mathcal T}], L^{2\alpha+2}(\Omega)) \quad \text{as} \quad m \rightarrow +\infty,
\end{align}
and there exist $v$ such that
 \begin{align}\label{con4}
|u_m|^{2\alpha}u_m \rightarrow v \quad \text{ weakly in } \quad L^{\frac{2\alpha+2}{2\alpha+1}}([0,{\mathcal T}],L^{\frac{2\alpha+2}{2\alpha+1}}(\Omega) ) \quad \text{as} \quad m \rightarrow +\infty.
\end{align}
Eventually, it holds
 \begin{align}\label{con5}
\frac{d}{dt} u_m \rightarrow \frac{d}{dt} u \quad \text{ weakly in } \quad L^2([0,{\mathcal T}], V_0^{'})+ L^{\frac{2\alpha+2}{2\alpha+1}}([0,{\mathcal T}],L^{\frac{2\alpha+2}{2\alpha+1}}(\Omega) ) \quad \text{as} \quad m \rightarrow +\infty.
\end{align}
Thus, thanks to Aubin-Lions compactness Lemma \cite{Lions,Temam}, we have
\begin{equation}\label{con7}
(u_m,\theta_m) \rightarrow (u,\theta) \quad \text{strongly \:in} \quad L^2([0,{\mathcal T}], H_0) \times   L^2([0,{\mathcal T}], H_1) \quad \text{as} \quad m \rightarrow +\infty.
\end{equation}
Now, let us fix $i$, such that $i<m$. We take the $H_0-$inner product of (\ref{eqpart})$_1$ with $\mathcal W_i$, and the $H_1-$inner product of  (\ref{eqpart})$_2$ with $\mathcal {\tilde W}_i$,  and integrate with respect to time over $[0,t]\subset [0, \mathcal T]$ to obtain
\begin{align*}
&  (u_m(t),\mathcal W_i)_{H_0} - (u_0,\mathcal {W}_i)_{H_0} +\nu \int_0^{t} \,(u_m(s),\mathcal W_i)_{\dot V_0}\,ds  \\ &+\int_0^{t} (B_0(u_m(s),u_m(s)) +a\, |u_m(s)|^{2\alpha}u_m(s)-\theta_m(s)\,{\bf e}_3, \mathcal W_i)_{H_0}\,ds= 0,\\
&    (\theta_m(t), \mathcal {\tilde W}_i)_{H_1}  -(\theta_0, \mathcal {\tilde W}_i)_{H_1} +\kappa \int_0^{t}\,(\theta_m(s), \mathcal {\tilde W}_i)_{V_1}\,ds\\ &+\int_0^{t} (B_1(u_m(s),\theta_m(s))- u_m\cdot {\bf e}_3, \mathcal {\tilde W}_i)_{H_1} \,ds=0.
\end{align*}
In the above we used the facts that $\mathcal{P}_m,\mathcal{\tilde P}_m$ and $\mathbb{P}$ are orthogonal projections in the corresponding $L^2$ spaces, and the fact that $\mathcal{P}_m \mathcal {W}_i = \mathcal {W}_i$ and $\mathcal{\tilde P}_m \, \mathcal {\tilde W}_i=\mathcal {\tilde W}_i$, since $i<m$.

Now, we can pass to the limit in these equalities.  The linear terms and the convective nonlinear terms are handled using   (\ref{con2}-\ref{con7}), and we refer to \cite{Temam,Temam1} for details. For the velocity power term, up to extracting a subsequence, we can show that $v=|u|^{2\alpha}u$ thanks to \eqref{con7} using Aubin-Lions compactness Lemma, \cite{Lions,Temam}. As a result we have shown that for every $i=1,2,\dots$  the limit pair $(u,\theta)$ satisfies
\begin{align*}
&  (u(t),\mathcal W_i)_{H_0} - (u_0,\mathcal {W}_i)_{H_0} +\nu \int_0^{t} \,(u(s),\mathcal W_i)_{\dot V_0}\,ds  \\ &+\int_0^{t} (B_0(u(s),u(s)) +a\, |u(s)|^{2\alpha}u(s)-\theta(s)\,{\bf e}_3, \mathcal W_i)_{H_0}\,ds= 0,\\
&    (\theta(t), \mathcal {\tilde W}_i)_{H_1}  -(\theta_0, \mathcal {\tilde W}_i)_{H_1} +\kappa \int_0^{t}\,(\theta(s), \mathcal {\tilde W}_i)_{V_1}\,ds\\ &+\int_0^{t} (B_1(u(s),\theta(s))- u\cdot {\bf e}_3, \mathcal {\tilde W}_i)_{H_1} \,ds=0.
\end{align*}
Consequently, the pair $(u,\theta)$ is a weak solution to \sys\, (see \cite{Temam,Temam1} for details).

\subsection{ Energy and gradient estimates}\label{weaksec}
In this section we prove the regularity of the solutions. For this purpose, we shall perform formal calculation using system \sys. We remark that these calculation can be justified rigorously by performing them first for the Galerkin approximate system and then pass to the limit as in the previous section. We start by revisiting the estimates we established in the previous section dedicated to  the existence of weak solutions. Specifically, we show the following

\begin{prop}\label{propexistweak}
Let $(u_0,\theta_0)\in H_0\times  H_1,\alpha\geq 0 $ and $a,\nu,\kappa> 0$. If $(u(t),\theta(t))$ is a weak solution of \sys, then
\begin{align*}
& u(t) \in  L^\infty(\R, H_0) \cap L^2_{\rm loc}(\R,V_0) \cap L^{2\alpha+2}_{\rm loc}(\R,L^{2\alpha+2}(\Omega)) \\
&\theta(t) \in L^\infty(\R, H_1) \cap L^2_{\rm loc}(\R,V_1).
\end{align*}
\end{prop}
\begin{proof}
Let $(u_0,\theta_0)\in H_0\times H_1$. We take the $H_0$-inner product of \sys$_1$ with $u$, and the $H_1$-inner product  of \sys$_2$ with $\theta$ to obtain, thanks to Cauchy-Schwarz, H\"older, and Young inequalities,
\begin{align}\label{dash1}
\frac12\,\frac{d}{dt}\, \nr u \nr^2_{H_0}+\nu\,\nr u\nr^2_{\dot V_0} +a\,\nr u\nr^{2\alpha+2}_{2\alpha+2}& \leq \epsilon\,\nr \theta\nr^2_{H_1}+\epsilon^{-1}\,\nr u\nr^2_{H_0},
\end{align}
and
\begin{align}\label{thissec}
\frac12\,\frac{d}{dt}\, \nr \theta \nr^2_{H_1}+\kappa\,\nr \theta\nr^2_{V_1} & \leq \epsilon\,\nr \theta\nr^2_{H_1}+\epsilon^{-1}\,\nr u\nr^2_{H_0},
\end{align}
respectively.
Observe that for all $\epsilon>0$, we have
\begin{equation}\label{int1}
\nr u\nr_{H_0}^2 \leq \epsilon\,\nr u\nr_{{2\alpha+2}}^{2\alpha+2} + 2 \epsilon^{-\frac1\alpha}\,L^2,
\end{equation}
Setting $\epsilon=\frac{2a\kappa\lambda}{a\kappa\lambda +32}$ in \eqref{int1},  and summing-up \eqref{dash1}  and \eqref{thissec} with $\epsilon=\frac{\kappa\lambda}{4}$, we obtain
\begin{align}\label{suminteg}
\nonumber \frac{d}{dt}\, \left( \nr u\nr^2_{H_0}+\nr \theta \nr^2_{H_1} \right)+2\nu\,\nr u\nr^2_{\dot V_0}+\frac\kappa2\,\nr \theta\nr^2_{V_1} +\frac a2 \nr u\nr^2_{H_0}&+ \frac{\kappa\lambda}{2}\,\nr \theta \nr^2_{H_1}+a\,\nr u\nr^{2\alpha+2}_{2\alpha+2}\\& \leq  \underbrace{2\,\left(\frac{2a\kappa\lambda}{a\kappa\lambda +32}\right)^{\frac{\alpha+1}{\alpha}}\,a\,L^2}_{:=\Gamma_0}.
\end{align}
Thanks to Gronwall's inequality, we get
\begin{align}\label{uthetah0bound}
\nonumber\nr u(t)\nr^2_{H_0}+\nr \theta(t) \nr^2_{H_1} &\leq \left(\nr u_0\nr^2_{H_0}+\nr \theta_0 \nr^2_{H_1} \right)\,e^{-\frac{\min{(a,\kappa\lambda})}2\,\,t} \\ &+\underbrace{\left(\frac{2a\kappa\lambda}{a\kappa\lambda +32}\right)^{\frac{\alpha+1}{\alpha}}\,\frac{4a\,L^2}{\min{(a,\kappa\lambda})}}_{:=\Gamma_1}\,(1-e^{-\frac{\min{(a,\kappa\lambda})}2\,\,t}).
\end{align}
Consequently, we have $u(t)\in L^\infty(\R; H_0)$ and $\theta(t) \in L^\infty(\R; H_0)$. In particular, we have
\begin{align}\limsup_{t\to+\infty}\,\left( \nr u(t)\nr^2_{H_0}+\nr \theta(t) \nr^2_{H_1}\right) \leq \Gamma_1.\label{limsupu}
\end{align}
Eventually, integrating \eqref{suminteg} with respect to time, and using \eqref{uthetah0bound} we get
\begin{align}\label{gradul2loc}
\int_0^t\,\nr u(s)\nr^2_{\dot V_0}\,ds \leq \frac{ \nr u_0\nr^2_{H_0}+\nr \theta_0 \nr^2_{H_1} }{2\nu} + \frac{\Gamma_0}{2\nu} t,
\end{align}
and
\begin{align*}
\int_0^t\,\nr \theta(s)\nr^2_{\dot V_1}\,ds \leq \frac{2\left(\nr u_0\nr^2_{H_0}+\nr \theta_0 \nr^2_{H_1} \right)}{\kappa} + \frac{2\Gamma_0}{\kappa}.
\end{align*}
In addition, we have
\begin{align}\label{l2alphaplus2local}
\int_0^t\,\nr u(s)\nr^{2\alpha+2}_{2\alpha+2}\,ds \leq \frac{ \nr u_0\nr^2_{H_0}+\nr \theta_0 \nr^2_{H_1} }{a} +  \frac{\Gamma_0}{a} t.
\end{align}
Therefore, we infer that $ u(t)\in L^2_{\rm loc}(\mathbb R^+;V_0)$, $\theta(t)  L^2_{\rm loc}(\mathbb R^+;V_1)$, and $L^{{2\alpha+2}}_{\rm loc}(\mathbb R^+;L^{{2\alpha+2}}(\Omega))$ .

\end{proof}
\noindent In the sequel, we need the following uniform Gronwall type lemma
\begin{lem}[\cite{Temam1}]\label{unigron}
Let $\varphi,\phi$ and $\psi$ be three non-negative locally integrable functions on $]t_0,+\infty[$ such that $\psi$ is absolutely continuous with $\psi'$ being local integrable on $]t_0,+\infty[$, and which satisfy
\[ \psi' \leq \varphi\,\psi +\phi, \quad \text{for all}\quad t\geq t_0,\]
and
\[\int_t^{t+s}\,\varphi(\tau)\,d\tau \leq a_1,\quad  \int_t^{t+s}\,\phi(\tau)\,d\tau \leq a_2, \quad \text{and}\quad \int_t^{t+s}\,\psi(\tau)\,d\tau \leq a_3,\quad \text{for all}\quad t\geq t_0,\]
where $s, a_1,a_2$ and $a_3$ denote positive constants. Then
\[ \psi(t+s) \leq \left(\frac{a_3}{s}+ a_2\right)\,e^{a_1},\quad \text{for all}\quad t\geq t_0.\]
\end{lem}
\noindent Now, we improve these weak solutions by considering more regular initial data for the velocity
\begin{prop}\label{propexiststrong}
Let $(u_0,\theta_0)\in V_0\times  H_1,\alpha> 1 $ and $a,\nu,\kappa> 0$. If $(u(t),\theta(t))$ is a weak solution of \sys, then in addition to the conclusion of Proposition \ref{propexistweak}, it holds  that $u(t) \in  L^\infty(\R, V_0) \cap L^2_{\rm loc}(\R,H^2(\Omega))$.
\end{prop}
\begin{proof}
First, we show that if  $(u_0,\theta_0)\in V_0\times  H_1$ and $\alpha>1$, then $ u(t)\in L^\infty_{\rm loc}(\R, V_0)$.  We take the $H_0$-inner product of \sys$_1$ with $-\Delta u$ to obtain
\begin{align*}
\frac12\,\frac{d}{dt}\,\nr  u\nr^2_{\dot V_0}+\nu\,\nr \Delta u\nr^2_2+\s\,(u\cdot\nabla)\,u\,\cdot\,(-\Delta\,u)\,dx &+a \,\nr |u|^\alpha\,\nabla u\nr^2_2 \leq \frac1\nu\,\nr \theta\nr^2_{H_1} + \frac\nu4\,\nr\Delta u\nr_2^2,
\end{align*}
since
\begin{align*}
-a\int_\Omega \Delta u \cdot |u|^{2\alpha} u\,dx&= a \int_\Omega \,|u|^{2\alpha}\,\,|\nabla u|^2dx + \frac{2a\alpha}{(\alpha+1)^2}\,\int_\Omega\,|\nabla |u|^{\alpha+1}|^2\,dx.
\end{align*}
Now, using Cauchy-Schwarz, H\"older and Young inequalities and assuming $\alpha>1$, we can write for all $\epsilon,\epsilon_0>0$
\begin{align}\label{cumbersome}
\left|\s\,(u\cdot \nabla)u \,\cdot\,(-\Delta\,u)\,dx \right| &\leq \s\,|u|\,|\nabla\,u|^\frac1\alpha\,|\nabla\,u|^{1-\frac1\alpha}\,|\Delta\,u|\,dx \nonumber\\
&\leq \nr |u|\,|\nabla\,u|^\frac1\alpha\nr_{{2\alpha}}\,\nr |\nabla\,u|^{1-\frac1\alpha}\nr_{\frac{2\alpha}{\alpha-1}}\,\nr\Delta\,u\nr_{2}\nonumber\\
&\leq \frac1{4\epsilon_0}\, \nr|u|^{\alpha}\,|\nabla\,u|\nr^{\frac2\alpha}_{2}\,\nr u\nr^{2(1-\frac1\alpha)}_{\dot V_0}  +\epsilon_0\nr\Delta\,u\nr^2_{2}\nonumber\\
&\leq \frac\epsilon{4\epsilon_0\,}\, \nr|u|^{\alpha}\,|\nabla u|\nr^{2}_{2} + \frac{\epsilon^{\frac1{1-\alpha}}}{4\epsilon_0} \,\nr u\nr^{2}_{\dot V_0}  +\epsilon_0\,\nr\Delta\,u\nr^2_{2}.
\end{align}
Optimizing in the $\epsilon'$s, we obtain
\begin{align}\label{toconwithgron}
\frac{d}{dt}\,\nr  u\nr^2_{\dot V_0}+\nu\,\nr \Delta u\nr^2_2 +a \,\nr |u|^\alpha\,\nabla u\nr^2_2 \leq \frac{1}\nu\,\nr \theta(t)\nr^2_{H_1} + \underbrace{\left(\frac{a\nu^\alpha}{2^{2-\alpha}}\right)^{1/1-\alpha}}_{:=\Gamma_2}\,\nr u\nr^2_{\dot V_0}.
\end{align}
Integrating this inequality with respect to time and using the fact that $\theta \in L^\infty(\R, H_1)$ and $ u(t)\in L^2_{\rm loc}(\mathbb R^+;V_0)$, we obtain that $\Delta u \in L^2_{\rm loc}(\mathbb R^+;V_0)$, hence $u\in L^2_{\rm loc}(\R, H^2(\Omega))$. Also, we have
\begin{equation}\label{powterminteg}
 |u|^\alpha\,\nabla u  \in  L^2_{\rm loc}(\R,L^2(\Omega)).
\end{equation}
Furthermore,  using the fact that  $\theta \in L^\infty(\R,H_1)$ then by means of the Gronwall's inequality we obtain that $u(t)\in L^\infty_{\rm loc}(\R, V_0)$.
\vskip6pt
\noindent Now, we show that actually $u(t)\in L^\infty(\R, V_0)$. We proceed in two steps, first for all $t \in[0,1]$ and then for all $t\geq 0$.
\vskip6pt
\noindent On the one hand, let $t \in[0,1]$, integrating \eqref{toconwithgron}, using \eqref{uthetah0bound} along with \eqref{gradul2loc}, we obtain
\begin{align*}
\nr  u(t)\nr^2_{\dot V_0} &\leq \nr  u_0\nr^2_{\dot V_0} + \frac{\nr u_0\nr^2_{H_0}+\nr \theta_0 \nr^2_{H_1}}{\nu}  +\frac{\Gamma_1}{\nu} + \frac{\Gamma_2}{2\nu}\,\left({ \nr u_0\nr^2_{H_0}+\nr \theta_0 \nr^2_{H_1} } + {\Gamma_0}\right).
\end{align*}
On the other hand, let $t\ge 0$ and we set
\[\psi=\nr u(t)\nr^2_{\dot V_0},\quad  \varphi =  \Gamma_2, \quad \text{and}\quad \phi=\frac{1}\nu\,\nr \theta(t)\nr^2_{H_1}. \]
Now, integrating  \eqref{suminteg} over $[t,t+1]$  we obtain  for all $t\geq 0$
\[
\int_t^{t+1}\,\nr u(\tau)\nr^2_{\dot V_0}\,d\tau \leq  \frac{1}{2\nu} \left({ \nr u_0\nr^2_{H_0}+\nr \theta_0 \nr^2_{H_1} } + {\Gamma_0+\Gamma_1}\right),
\]
and
\[
\int_t^{t+1}\,\nr \theta(\tau)\nr^2_{H_1}\,d\tau \leq  \frac{2}{\kappa\lambda} \left({ \nr u_0\nr^2_{H_0}+\nr \theta_0 \nr^2_{H_1} } + {\Gamma_0+\Gamma_1}\right).
\]
Eventually, thanks to Lemma \ref{unigron} with s=1, we obtain for all $t\geq 1$
\begin{align*}
\nonumber\nr u(t)\nr_{\dot V_0}^2 &\leq \frac{\kappa\lambda+4}{2\kappa\lambda\nu} \left({ \nr u_0\nr^2_{H_0}+\nr \theta_0 \nr^2_{H_1} } + {\Gamma_0+\Gamma_1}\right)\,\exp{(\Gamma_2)}.
\end{align*}
All in all, we get $u(t) \in  L^\infty(\R, V_0) $.
\end{proof}
\noindent Now, we show the existence of absorbing ball in $V_0$ for the velocity.  For this purpose, we let $t\ge1$  and $s\in [t-1,t]$, we integrate \eqref{toconwithgron} over $[s,t]$ and get
\begin{align}\label{startofit}
\nr  u(t)\nr^2_{\dot V_0} \leq  \nr  u(s)\nr^2_{\dot V_0} + \frac{1}\nu\,\int_{t-1}^t\,\nr \theta(\tau)\nr^2_{H_1}\,d\tau + \Gamma_2\,\int_{t-1}^t\,\nr u(\tau)\nr^2_{\dot V_0}\,d\tau.
\end{align}
Now, on the one hand, integrating \eqref{uthetah0bound}, we obtain
\begin{align}\label{temp2fortofit0}
\frac{1}\nu\,\int_{t-1}^t\,\nr \theta(\tau)\nr^2_{H_1}\,d\tau &\leq \frac{2\left(\nr u_0\nr^2_{H_0}+\nr \theta_0 \nr^2_{H_1}\right)}{\nu\,\min{(a,\kappa\lambda)}}\,\left(e^{\frac{\min{(a,\kappa\lambda})}2}-1\right)\,e^{-\frac{\min{(a,\kappa\lambda})}2\,\,t} +\frac{\Gamma_1}{\nu}.
\end{align}
On the other hand, going back to \eqref{dash1}, and using \eqref{int1}, we get
\begin{align*}
\frac{d}{dt}\, \nr u \nr^2_{H_0}+2\,\nu\,\nr u\nr^2_{\dot V_0} + a\nr u\nr^{2}_{H_0}& \leq a\,\nr \theta\nr^2_{H_1} +  4aL^2 .
\end{align*}
Integrating this inequality over $[t-1,t]$, and use \eqref{temp2fortofit0} to obtain
\begin{align}\label{temp2fortofit}
\nonumber&\int_{t-1}^t\,\nr u(\tau)\nr^2_{\dot V_0}\,d\tau \leq \frac1{2\nu}\,\nr u(t-1)\nr^2_{H_0}+ \frac a{2\nu}\,\int_{t-1}^t\,\nr \theta(\tau)\nr^2_{H_1}\,d\tau +\frac{2aL^2}{\nu} \\
&\hskip30pt\leq \frac1{2\nu}\,\nr u(t-1)\nr^2_{H_0} +\frac{a\left(\nr u_0\nr^2_{H_0}+\nr \theta_0 \nr^2_{H_1}\right)}{\nu\,\min{(a,\kappa\lambda)}}\,\left(e^{\frac{\min{(a,\kappa\lambda})}2}-1\right)\,e^{-\frac{\min{(a,\kappa\lambda})}2\,\,t} +\frac{a\,\Gamma_1 +4 aL^2}{2\nu}.
\end{align}
Substituting \eqref{temp2fortofit0} and \eqref{temp2fortofit} in \eqref{startofit}, we get  for $s\in[t-1,t]$
\begin{align*}
\nonumber \nr  u(t)\nr^2_{\dot V_0} \leq \nr  u(s)\nr^2_{\dot V_0}&+\frac{\Gamma_2}{2\nu}\,\nr u(t-1)\nr_{H_0}^2 +\frac{2+a\Gamma_2}{\nu\min{(a,\kappa\lambda)}}\left(\nr u_0\nr^2_{H_0}+\nr \theta_0 \nr^2_{H_1}\right)\,\left(e^{\frac{\min{(a,\kappa\lambda})}2}-1\right)\,e^{-\frac{\min{(a,\kappa\lambda})}2\,\,t}\\&+\frac{\Gamma_2\left(a\Gamma_1 + 4aL^2\right)+2\Gamma_1}{2\nu}.
\end{align*}
Thanks to \eqref{uthetah0bound}, we have, for $s\in[t-1,t]$
\begin{align*}
\nonumber \nr  u(t)\nr^2_{\dot V_0} \leq \nr  u(s)\nr^2_{\dot V_0}&+\frac{2+a\Gamma_2}{\nu\min{(a,\kappa\lambda)}}\left(\nr u_0\nr^2_{H_0}+\nr \theta_0 \nr^2_{H_1}\right)\,\left(e^{\frac{\min{(a,\kappa\lambda})}2}-1\right)\,e^{-\frac{\min{(a,\kappa\lambda})}2\,\,t}\\
&+\frac{\Gamma_2\left[(a+1)\Gamma_1 + 4aL^2\right]+2\Gamma_1}{2\nu}.
\end{align*}
Integrating this inequality with respect to $s$ over $[t-1,t]$, we get
\begin{align*}
\nr  u(t)\nr^2_{\dot V_0} \leq \int^t_{t-1}\nr  u(s)\nr^2_{\dot V_0}\,ds&+\frac{2+a\Gamma_2}{\nu\min{(a,\kappa\lambda)}}\left(\nr u_0\nr^2_{H_0}+\nr \theta_0 \nr^2_{H_1}\right)\,\left(e^{\frac{\min{(a,\kappa\lambda})}2}-1\right)\,e^{-\frac{\min{(a,\kappa\lambda})}2\,\,t}\\
&+\frac{\Gamma_2\left[(a+1)\Gamma_1 + 4aL^2\right]+2\Gamma_1}{2\nu}.
\end{align*}
Now, using again \eqref{temp2fortofit}, we get
\begin{align*}
\nr  u(t)\nr^2_{\dot V_0} \leq \frac1{2\nu}\,\nr u(t-1)\nr^2_{H_0} &+\frac{2+a(\Gamma_2+1)}{\nu\min{(a,\kappa\lambda)}}\left(\nr u_0\nr^2_{H_0}+\nr \theta_0 \nr^2_{H_1}\right)\,\left(e^{\frac{\min{(a,\kappa\lambda})}2}-1\right)\,e^{-\frac{\min{(a,\kappa\lambda})}2\,\,t}\\
&+\frac{\Gamma_2\left[(a+1)\Gamma_1 + 4aL^2\right]+(2+a)\Gamma_1 + 4aL^2}{2\nu}.
\end{align*}
Eventually, using \eqref{uthetah0bound}, the latter inequality gives
\begin{align*}
\nr  u(t)\nr^2_{\dot V_0} \leq &+\frac{2+a(\Gamma_2+1)}{\nu\min{(a,\kappa\lambda)}}\left(\nr u_0\nr^2_{H_0}+\nr \theta_0 \nr^2_{H_1}\right)\,\left(e^{\frac{\min{(a,\kappa\lambda})}2}-1\right)\,e^{-\frac{\min{(a,\kappa\lambda})}2\,\,t}\\
&+\frac{\Gamma_2\left[(a+1)\Gamma_1 + 4aL^2\right]+(3+a)\Gamma_1 + 4aL^2}{2\nu}.
\end{align*}
Threfore, one has
\begin{align}\label{limsupgradthmstraong1}
\limsup_{t\to+\infty} \,\nr  u(t)\nr^2_{\dot V_0} &\leq \frac{\Gamma_2\left[(a+1)\Gamma_1 + 4aL^2\right]+(3+a)\Gamma_1 + 4aL^2}{2\nu}.
\end{align}
\vskip6pt
\noindent Next, we have
\begin{prop}\label{proptimederiv}
Let $(u_0,\theta_0)\in V_0\cap L^{2\alpha+2}(\Omega) \times  H_1,\alpha> 1 $ and $a,\nu,\kappa> 0$. If $(u(t),\theta(t))$ is a weak solution of \sys, then in addition to the conclusion of Proposition \ref{propexiststrong}, it holds  that $u\in L^{\infty}_{\rm loc}(\mathbb R^+;L^{{2\alpha+2}}(\Omega))$ and $\partial_t u\in L^2_{\rm loc}(\mathbb R^+; L^2(\Omega))$.
\end{prop}
\begin{proof}
We multiply\sys$_1$ by $\partial_tu$ and integrate over $\Omega$ to get thanks to the Cauchy-Schwarz and Young inequalities
\begin{align*}
\nr\partial_t u\nr_2^2 + \s\,(u\cdot \nabla)u\cdot \partial_t u\,dx + \frac\nu2\,\frac{d}{dt}\nr \nabla u\nr_2^2 &+\frac{a}{2\alpha+2}\,\frac{d}{dt}\,\nr u\nr_{2\alpha+2}^{2\alpha+2}  \leq \nr \theta\nr_{H_1}^2 +\frac14\,\nr\partial_tu\nr_2^2.
\end{align*}
Now, proceeding as in \eqref{cumbersome}, we have
\begin{align}\label{toseeitlater}
\s\,(u\cdot \nabla)u\cdot \partial_t u\,dx &\leq \nr|u|^{\alpha}\,|\nabla u|\nr^{2}_{2} + \,\nr u\nr^{2}_{\dot V_0}  +\frac14\,\nr\partial_t\,u\nr^2_{2}.
\end{align}
Thus, we have
\begin{align*}
\int_0^t\nr\partial_s u(s)\nr_2^2\,ds\,  + \frac{a}{2\alpha+2}\,\nr u (t)\nr_{2\alpha+2}^{2\alpha+2}& \leq \frac\nu2\,\nr u_0\nr^2_{\dot V_0}+\frac{a}{\alpha+1}\,\nr u_0\nr_{2\alpha+2}^{2\alpha+2} \\
&+ 2\,\int_0^t\,\left(\nr \theta(s)\nr_{H_1}^2+\nr|u(s)|^{\alpha}\,|\nabla u(s)|\nr^{2}_{2} + \,\nr u(s)\nr^{2}_{\dot V_0}\right)\,ds.
\end{align*}
Using \eqref{uthetah0bound},\eqref{gradul2loc}, \eqref{powterminteg}, and Proposition \ref{propexiststrong}, we obtain for all $(u_0,\theta_0)\in { V_0}\cap L^{2\alpha+2}(\Omega) \times H_1$
\begin{equation*} u\in L^{\infty}_{\rm loc}(\mathbb R^+;L^{{2\alpha+2}}(\Omega)),\quad \text{and}\quad \partial _t u\in L^2_{\rm loc}(\mathbb R^+; L^2(\Omega)).\end{equation*}
\end{proof}
\begin{remark}\label{tempinfinity}
Let $(u_0,\theta_0)\in H_0 \times  L^\infty(\Omega), \alpha> 1 $ and $a,\nu,\kappa> 0$. If $(u(t),\theta(t))$ is a weak solution of \sys, then $\theta(t) \in L^{\infty}(\Omega)$. This property can be readily shown using a variant of the maximum principle as in \cite{Temam1} applied to the original system $(\mathcal {S}_o-\mathcal{BC}_o)$ first to obtain that $T(t)\in L^\infty (\Omega)$ if $T(t=0) \in L^\infty (\Omega)$. Next, one uses the relation \eqref{change} together with the fact that $-1\leq z\leq 1$ to obtain the desired result.
\end{remark}
\subsection{Continuous dependence on the initial data and uniqueness of solutions}\label{uniquenesssection}
In this section, we prove the continuous dependence of the weak solutions on the initial data with respect to the  $H_0 \times H^{-1}(\Omega)$ topology, in particular their uniqueness. For this purpose, let $(u,\theta)$ and $(v,\eta)$ be two weak solutions of system \sys, and let $w=u-v$ and $\xi=\theta-\eta$. It is rather clear that $(w,\xi)$ enjoys
\begin{align*}
\mathcal S_d:\quad \left\lbrace\begin{array}{ll}
&\partial_t\,w+\nu\,\mathcal A_0\,w + (v\cdot\nabla)w+(w\cdot\nabla)u +a\, |u|^{2\alpha}u-a\, |v|^{2\alpha}v +\nabla \pi=  \xi\,{\bf e}_3,\\ \\
&\partial_t\,\xi+\kappa\,\mathcal A_1\,\xi +(v\cdot\nabla)\xi+(w\cdot\nabla)\theta = w\cdot {\bf e}_3,\\ \\
& \nabla\cdot w=0,\;w_{\vert_{t=0}}=w_0=u_0-v_0,\; \xi_{\vert{{t=0}}}=\xi_0=\theta_0-\eta_0,\\
\end{array}
\right.
\end{align*}
with the associated periodic BC obtained from $\mathcal {BC}_e$ satisfied by $(u,\theta)$ and $(v,\eta)$. In the sequel, we will need the following  strong monotonicity form (see, {\it e.g.}, \cite{Barret}): There exists a positive constant $\delta(\alpha)$ such that
\begin{equation}\label{positivity}
\delta\,|u-v|^{2}\,\left(|u|+|v|\right)^{2\alpha}\leq \left(|u|^{2\alpha}u-|v|^{2\alpha}v\right)\cdot (u-v) .
\end{equation}
In order to show the uniqueness, first we take the $H_0$-inner product of $(\mathcal {S}-\mathcal{BC}_e)_1$ with $w$ and  integrate over $\Omega$ to obtain
\begin{align}\label{uniquenesspart1}
\frac12\,\frac{d}{dt}\, \nr w \nr^2_{H_0} &+\nu\,\nr  w\nr^2_{\dot V_0} + \int_\Omega\,(w\cdot\nabla)u\cdot w \,dx+ a\,\s |u|^{2\alpha}u\,\cdot w\,dx-a\,\s |v|^{2\alpha}v\cdot  w\,dx\\\nonumber &= \s \xi\,(w\cdot {\bf e_3})\,dx.
\end{align}
Next, let $\psi_1$ and $\psi_2$ be the unique periodic functions solutions of the elliptic equations $\theta=\Delta \psi_1$ and $\eta=\Delta \psi_2$ satisfying $\int _\Omega \psi_1dx=\int _\Omega \psi_1dx=0$. Since $\theta$ and $\eta$  are odd functions with respect of the $z-$variable then $\psi_1$ and $\psi_2$ enjoy the same property. Now, let $\psi=\psi_1-\psi_2$, therefore $\psi$ satisfies the following equation
\[\partial_t\,\Delta \psi+\kappa\,\mathcal A_1\,\Delta\Psi +(v\cdot\nabla)\Delta \psi+(w\cdot\nabla)\Delta\psi_1 = w\cdot {\bf e}_3\]
Taking the duality action of the latter equation with $\psi$ gives
\begin{align}\label{uniquenesspart2}
\frac12\,\frac{d}{dt}\, \nr \nabla \psi \nr^2_{2} &+\kappa\,\nr \Delta\psi\nr^2_{2} + \int_\Omega\,(v\cdot\nabla)\Delta \psi\,\psi\,dx  +\int_\Omega\,(w\cdot\nabla)\Delta \Psi_1\,\psi\,dx = \s \psi\,(w\cdot {\bf e_3})\,dx.
\end{align}
Now, using Cauchy-Schwarz and Young inequalities, we can write
\begin{align}\label{threediffbasic}
 \s \xi\,(w\cdot {\bf e_3})\,dx=  \s \Delta \psi\,(w\cdot {\bf e_3})\,dx &\leq \nr \Delta \psi\nr_{2}\,\nr w\nr_{H_0}  \leq \frac\kappa3\nr \Delta \psi\nr^2_{H_1} + \frac 3\kappa\,\nr w\nr^2_{H_0}.
 \end{align}
 Next, since  $\nabla \cdot w=0$ then by  integrating by parts and using the  generalized H\"older inequality, we have  for all $\alpha>1$
 \begin{align*}
|\int_\Omega\,(w\cdot\nabla)u\cdot  w \,dx |= |-\int_\Omega\,(w\cdot\nabla)w\cdot  u \,dx|
&  \leq \s |w|\,|u|\,|\nabla w|\,dx|= \s\,|u|\,|w|^\frac1\alpha\,|w|^{1-\frac1\alpha}\,|\nabla\,w|\,dx \nonumber\\
&\leq \nr |u|\,|w|^\frac1\alpha\nr_{{2\alpha}}\,\nr |w|^{1-\frac1\alpha}|\nr_{\frac{2\alpha}{\alpha-1}}\,\nr w\nr_{\dot V_0}\nonumber\\
&\leq \frac2{\nu}\, \nr|u|^{\alpha}\,|w|\nr^{\frac2\alpha}_{\0}\,\nr w\nr^{2(1-\frac1\alpha)}_{H_0}  +\frac\nu2\,\nr w\nr^2_{\dot V_0}\nonumber\\
&\leq a\delta\, \nr|u|^{\alpha}\,w\nr^{2}_{2} +(a\delta\nu^\alpha)^{\frac1{1-\alpha}}\,\nr w\nr^{2}_{H_0} +\frac\nu2\,\nr w\nr^2_{\dot V_0}.
\end{align*}
Similarly, since  $\nabla \cdot w=0$ once again we integrate by parts and use the H\"older and Gagliardo-Nirenberg inequalities we deduce
\begin{align*}
|\int_\Omega\,(w\cdot\nabla)\Delta \psi_1\,\psi\,dx| &=|-\int_\Omega\,(w\cdot\nabla)\psi\,\Delta\psi_1\,dx| \\
\nonumber &\leq \nr w\nr_6\,\nr \nabla \psi \nr_3\,\nr\Delta\psi_1\nr_{2}\, \\
 \nonumber& \leq  \nr w\nr_6\,\nr \Delta \psi \nr_2^{\frac12}\,\nr \nabla \psi \nr_2^{\frac12}\,\nr\theta\nr_{H_1} \\
 \nonumber & \leq \frac\nu2\nr w\nr^2_6 + \frac2\nu\nr \Delta \psi \nr_2\,\nr \nabla \psi \nr_2\,\nr\theta\nr^2_{H_1}  \\
  \nonumber & \leq \frac\nu2\nr \nabla w\nr^2_2 + \frac\kappa3 \nr\Delta \psi \nr^2_2 + \frac{12}{\kappa\nu^2}\nr \nabla \psi \nr^2_2\,\nr\theta\nr^4_{H_1} \\
   \nonumber & \leq \frac\nu2\nr \nabla w\nr^2_2 + \frac\kappa3 \nr\Delta \psi \nr^2_2 + \frac{c}{\kappa\nu^2}\nr \nabla \psi \nr^2_2,
\end{align*}
where $c$ is a constant depending on $\nr u_0\nr_{H_0}$ and $\nr\theta_0\nr_{H_1}$ (see \eqref{uthetah0bound}).
Also, we have the following
\begin{align*}
\int_\Omega\,(v\cdot\nabla)\Delta \psi\,\psi\,dx &=-\int_\Omega\,(v\cdot\nabla)\psi\,\Delta\psi\,dx \\
\nonumber &\leq \nr v\nr_6\,\nr \nabla \psi \nr_3\,\nr\Delta\psi\nr_{2}\, \\
&\leq \nr v\nr_{\dot V_0}\,\,\nr \Delta \psi \nr_2^{\frac32}\,\nr \nabla \psi \nr_2^{\frac12}\, \\
&\leq \frac\kappa3\nr \Delta \psi \nr_2^{2} + \frac3\kappa \nr v\nr^2_{\dot V_0}\,\nr \nabla \psi \nr_2^{2}
\end{align*}
Thanks to \eqref{positivity}, we have
\begin{align*}
\delta\,\nr |u|^\alpha\,w\nr^2_2\leq \delta\,\nr (|u|+|v|)^\alpha\,w\nr^2_2\leq \s \left(|u|^{2\alpha}u-|v|^{2\alpha}v\right)\,\cdot w\,dx.
\end{align*}
Summing-up \eqref{uniquenesspart1} and \eqref{uniquenesspart2} and collecting the estimates above, there exist $c_1>0$ depending only on $\alpha,\nu$ and $\kappa$ and $c_2>0$ depending on  $\alpha,\nu, \kappa$ and the $L^2$ norm of the initial data such that
\begin{align*}
\frac{d}{dt}\, \left\{\nr w \nr^2_{H_0}+ \nr \xi \nr^2_{H^{-1}}\right)  \leq   \max{\left\{c_1+c_2+ \nr v\nr^2_{\dot V_0} \right\}}\,\left(\nr w\nr^2_{H_0}+ \nr \xi\nr^2_{H^{-1}}\right).
\end{align*}
Eventually, Gronwall's Lemma combined to the fact  that $v(t)\in L^2_{\rm loc}(\mathbb R^+;V_0)$ (see \eqref{gradul2loc}), shows the continuous dependence of the solutions on the initial data in $L^\infty([0,\mathcal T],H_0 \times H^{-1}(\Omega)$, in particular their uniqueness.
\begin{remark}\label{strongtheta}
Let us mention that Theorem \ref{thmstrong} can be easily extended to  the case where the initial data $(u_0,\theta_0) \in V_0\times V_1$ where it can be shown that the solution satisfies
\begin{align*}
&u(x,t) \in C_b^0(\mathbb R^+;V_0)\cap L^2_{\rm loc}(\mathbb R^+; V_0\cap H^2(\Omega)) \cap L^{{2\alpha+2}}_{\rm loc}(\mathbb R^+;L^{{2\alpha+2}}(\Omega)),\\&{\rm and}\\
& \theta(x,t) \in C_b^0(\mathbb R^+;V_1)\cap L^2_{\rm loc}(\mathbb R^+; V_1\cap H^2(\Omega)).
\end{align*}
The proof of the uniqueness given in section \ref{uniquenesssection} remains unchanged. Indeed, the necessary estimate is obtained by taking the the $V_1$-inner product of \sys$_2$ with $-\Delta \theta$ to obtain
\begin{align*}
\frac12\,\frac{d}{dt}\,\nr  \theta \nr^2_{V_1}+\kappa\,\nr \Delta \theta\nr^2_2+\s\,(u\cdot\nabla)\,\theta\,\cdot\,(-\Delta\,\theta)\,dx  = \s u_3\,\Delta \theta \,dx \end{align*}
Now,
\[ \s u_3\,\Delta \theta \,dx \leq \frac\kappa4 \nr \Delta \theta\nr^2_2 + \frac4\kappa \nr u\nr^2_{H_0} \leq  \frac\kappa4 \nr \Delta \theta\nr^2_2 + {\rm Const.}\]
and (for instance)
\begin{align*}
\s\,(u\cdot\nabla)\,\theta\,\cdot\,(-\Delta\,\theta)\,dx &\leq \nr \nabla w\nr^2_2 + \frac\kappa4 \nr\Delta \psi \nr^2_2 + \frac\kappa4 \nr \nabla \psi \nr^2_2\leq  {\rm Const.} + \frac\kappa4 \nr\Delta \psi \nr^2_2+ \frac\kappa4 \nr \nabla \psi \nr^2_2.
\end{align*}
Collecting these estimates, and using \eqref{thissec}, one can show immediately that $\theta \in L^\infty(\R, V_1)$. Furthermore, it can be also shown the existence of absorbing ball in  $V_0\times V_1$ following the same argument used for the velocity above.
\end{remark}
\section{Data Assimilation}\label{dasec}
Data assimilation belongs to the family  of semi-empirical methods that aim to enhance the prediction's quality of the physical phenomena at hand, by synchronizing information collected from measurements at coarse spatial scales with the theoretical models. This method was originally proposed for atmospheric predictions such as weather forecasting in \cite{daley1993atmospheric}. A new approach was introduced in \cite{azouani2013feedback} based on ideas from control theory (see also \cite{olson2003determining,olson2008determining}) consisting in the introduction of a feedback control term that nudges the large spatial scales of the model towards those of the reference solution. In the context of data assimilation, the large spatial scales of the reference solution are constructed by an interpolation operator from the spatial coarse scale measurements. This new approach was applied for several models including the $3D$ Navier-Sotokes-$\alpha$ in \cite{albanez2016continuous}, the  $2D$  B\'enard-convection where only measurement of velocity is needed in \cite{farhat2015continuous}, the $2D$ Navier-Stokes where the authors show the convergence with measurements of a single velocity component in \cite{farhat2016abridged},  the $3D$ Brinkman-Forchheimer-extended Darcy model in \cite{MTT} etc. Other features of this new approach were developed and investigated in several recent publications such as \cite{biswas2019downscaling,desamsetti2019efficient,garcia2020uniform,gesho2016computational,ibdah2020fully,jolly2019continuous,mondaini2018uniform} and references therein. In this paragraph, we present the data assimilation algorithm applied to our system \sys \, and state the result that can be readily obtained  by combining Theorems \ref{thmweak} and \ref{thmstrong} and techniques from \cite{farhat2017continuous}. Specifically, we state the convergence of the algorithm in the case of system $(\mathcal {S}-\mathcal{BC})$ with measurement of only two components of the velocity. Before going further, let us present the idea of the algorithm. Briefly speaking, the observation are introduced into the system $(\mathcal {S}-\mathcal{BC})$ by the mean of an interpolant operator $\mathcal I_h(u(t))$ where $u(t)$ denotes the velocity part of the solution $(u(t),\theta(t))$. This operator interpolates the observations of the system $(\mathcal {S}-\mathcal{BC})$ at coarse scale spatial resolution of size $h$ to be specified. We shall consider two types of interpolants on the space domain $\Omega_r:= [0,L]\times [0,L] \times[0,1]$:
\begin{align}\label{interpolant1}
\mathcal I_h:\, H^1(\Omega_r)\rightarrow L^2(\Omega_r) &\quad \text{satisfying}\quad \nr \psi-\mathcal I_h(\psi)\nr^2_\0 \leq c_0\,h^2\,\nr \nabla\,\psi\nr^2_{\0},
\end{align}
 for all vectors $\psi=(\psi_1,\psi_2)\in (H^1(\Omega_r))^2$, and
\begin{align}\label{interpolant2}
\mathcal I_h:\, H^2(\Omega_r)\rightarrow L^2(\Omega_r) &\quad \text{satisfying}\quad \nr \psi-\mathcal I_h(\psi)\nr^2_\0 \leq c_0\,h^2\,\nr \nabla\,\psi\nr^2_{\0}+ c_1\,h^4\nr \Delta\,\psi\nr^2_{\0},
\end{align}
for all vectors $\psi=(\psi_1,\psi_2)\in (H^2(\Omega_r))^2$,  where $c_0$ and $c_1$ denote dimensionless nonnegative constants. We refer the reader to  \cite{albanez2016continuous,azouani2014continuous,MTT} for physical examples of such interpolants. To simplify our notation, we will denote $v_\perp:=(v_1,v_2)^T$ for all three-components vector $v=(v_1,v_2,v_3)^T$. Assume that the initial data $(u_0,\theta_0)$ of system $(\mathcal {S}-\mathcal{BC})$ is missing. The idea is to recover this initial data as accurately as possible. This will be achieved by constructing a solution $(v(t),\eta(t))$ from the observations that satisfies
\begin{align*}
\mathcal S_{da} :\quad \left\lbrace\begin{array}{ll}
&\partial_t\,v_\perp-\nu\,\Delta\,v_\perp +B_0(v,v_\perp) +a\,|v|^{2\alpha}v_\perp+\nabla_\perp\,q=\mu\,\mathcal { I}_h(u_\perp-v_\perp),\\ \\
&\partial_t\,v_3-\nu\,\Delta\,v_3 +B_0(v,v_3) +a\,|v|^{2\alpha}v_3+\partial_3\,q = \eta \,{\bf e}_3,\\ \\
&\partial_t\,\eta-\kappa\,\Delta\,\eta +B_1(v,\eta) = v\cdot{\bf e}_3,\\ \\
& \nabla\cdot v=0,\;u_{\vert_{t=0}}=v_0,\; \eta_{\vert{{t=0}}}=\eta_0,\\
\end{array}
\right.
\end{align*}
 subjected to the set of boundary conditions $\mathcal {BC}$ (with $(u,\theta,p)$ replaced by $(v,\eta,q)$).
 Observe that we do not use measurements of the temperature and the third component of the velocity. The parameter $\mu>0$ is the nudging parameter and $h$ is the resolution parameter of the collected data. To avoid technical difficulties, one proceeds as for the original system $(\mathcal {S}-\mathcal{BC})$ by extending the domain from $\Omega_r$ to $\Omega$ and therefore extending the BC to $\mathcal {BC}_e$. It is important to notice that the operator $\mathcal I_h$ should be as well extended to act on functions in  $H^1(\Omega)$ and $H^1(\Omega)$ since the collected measurement are on the reference solution of $(\mathcal {S}-\mathcal{BC})$. Specifically, its extension should be even since it acts only on the first two components of the velocity.
 \vskip6pt
\noindent The extended data assimilation system  \sysss\, can be shown to admit unique weak and strong solutions equivalently to \sys. Most importantly, it can be shown that for a space resolution $h$ small enough and nudging parameter $\mu$ large enough the weak solution of $(v(t),\eta(t))$ of \sysss\, converges asymptotically in time at an exponential rate to the reference solution $(u(t),\theta(t))$ of system \sys\, in $H_0\times H^{-1}(\Omega)$ for interpolants of type \eqref{interpolant1} for all $\alpha>1$. Equivalently, the strong solution $(v(t),\eta(t))$ of \sysss\, converges asymptotically in time at an exponential rate to the reference solution $(u(t),\theta(t))$ of system \sys\, in $V_0\times H^{-1}(\Omega)$ for interpolants of type \eqref{interpolant2} but only for $1<\alpha<2$.

\vskip6pt
\noindent The proof is based on a combination of the arguments presented in the previous sections and ideas from \cite{farhat2016data,MTT}.

\section*{Acknowledgments}
\noindent This publication was made possible by NPRP grant\# S-0207-200290 from the Qatar National Research Fund (a member of Qatar Foundation). The findings herein reflect the work, and are solely the responsibility, of the authors. E.S.T. would like to thank the Isaac Newton Institute for Mathematical Sciences, Cambridge, for support and hospitality during the programme ``Mathematical aspects of turbulence: where do we stand?" where part of the work on this paper was undertaken. This work was supported in part by EPSRC grant no EP/R014604/1.

\bibliographystyle{siam}
\bibliography{Boussinesq-Titi-Trab}

\end{document}